\documentclass[centertags,12pt]{amsart}

\usepackage{amssymb}
\usepackage{latexsym}
\usepackage{amsfonts}
\usepackage{amssymb}
\usepackage{graphicx} % adicionado por Matheus em 22/09
\usepackage{pdflscape} % adicionado por Matheus em 22/09
\usepackage{diagbox} % adicionado por Matheus em 22/09
\usepackage{color}
\usepackage{hyperref}
\usepackage[utf8]{inputenc}

\makeatletter
\renewcommand{\subsection}{\@startsection
{subsection}{2}{0mm}{\baselineskip}{-0.25cm}
{\normalfont\normalsize\em}}
\makeatother

\newtheorem{theorem}{Theorem}[section]
\newtheorem{proposition}[theorem]{Proposition}
\newtheorem{corollary}[theorem]{Corollary}
\newtheorem{lemma}[theorem]{Lemma}

{\theoremstyle{definition}

\newtheorem{example}[theorem]{Example}

}

{\theoremstyle{remark}
\newtheorem{remark}[theorem]{Remark}

}

\def\1{\mathbf 1}

\def\F{\mathbf F}
\def\PP{\mathbf P}

\def\T{\mathbf T}

\def\bZ{\mathbb Z}

\def\Ap{{\rm Ap}}
\def\Kunz{{\rm Kunz}}

\def\N {\mathbb{N}}
\def\C {\mathbb{C}}

\def\b {\beta}

\def\k {\kappa}
\def\l {\ell}

\title[Gapsets and the $k$-generalized Fibonacci sequences]{Gapsets and the $k$-generalized Fibonacci sequences}

\author[G. B. Almeida Filho]{Gilberto B. Almeida Filho}
\address{Instituto Federal de São Paulo - Campus São Paulo, São Paulo, SP, Brazil}
\email{gbrito.af19@gmail.com}

\author[M. Bernardini]{Matheus Bernardini}
\address{Faculdade do Gama, Universidade de Bras\'{i}lia, Bras\'{i}lia, DF, Brazil}
\email{matheusbernardini@unb.br}

\thanks{{\em 2020 Math. Subj. Class.}: Primary 20M14; 
Secondary 05A15, 05A19}


\begin{document}

%%%%%%%%%%%%%%%%%%%%%%%%%%%%%%%%%%%%%%%%%%%%%%%%%%%%%%%%%%%%%%%%%%%%%%%%%%%%%%%%%%%%%%%%%%%%%%%%%%%%%%%%%%%%%%%%%%%%%%%%%

\begin{abstract} 
In this paper, we bring the terminology of the Kunz coordinates of numerical semigroups to gapsets and we generalize this concept to $m$-extensions. It allows us to identify gapsets and, in general, $m$-extensions with tilings of boards. As a consequence, we prove a version of Bras-Amor\'{o}s conjecture for $m$-extensions. Besides, we obtain a lower bound for the number of gapsets with fixed genus and depth at most 3 and a family of upper bounds for the number of gapsets with fixed genus. Moreover, we present explicit formulas for the number of gapsets with fixed genus and depth, when the multiplicity is 3 or 4, and, in some cases, for the number of gapsets with fixed genus and depth.
%Those bounds are related to the $k$-generalized Fibonacci sequences, the Padovan sequence and quasipolynomials. 
\end{abstract}

\maketitle

\section{Introduction}\label{s1}

A \textit{numerical semigroup} is a cofinite submonoid of the set of non-negative integers, $\N_0$, equipped with the usual addition (see \cite{GS-R} for a background). The complement of a numerical semigroup $S$ is called the \textit{set of gaps} of $S$ and it is denoted by $G(S)$; its cardinality is the \textit{genus} of $S$, which is denoted by $g(S)$. Throughout this paper, we denote by $\N$ the set of positive integers, by $\bZ$ the set of integers, by $[a,b] := \{x \in \bZ: a \leq x \leq b\}$, for integers $a$ and $b$ and for a real number $x$, $\lceil x \rceil$ and $\lfloor x \rfloor$ are the smallest integer greater than or equal to $x$
and the biggest integer smaller than or equal to $x$, respectively. Some other invariants play a role in the theory and we introduce them as follows: the \textit{multiplicity}, the \textit{conductor} and the \textit{depth} of a numerical semigroup are \linebreak $m(S) := \min\{s \in S: s \neq 0\}$, $c(S) := \min\{s \in S: s + n \in S, \forall n \in \N_0\}$ and $q(S) := \left \lceil \frac{c(S)}{m(S)} \right \rceil$, respectively. A central problem in numerical semigroup theory is the problem of counting numerical semigroups by genus (see \cite{Kaplan2} for a survey on this subject). One can prove that if $S$ is a numerical semigroup of genus $g$, then $S$ contains every integer greater than $2g - 1$; hence, the number of numerical semigroups of genus $g$, namely $n_g$, is finite. The sequence $(n_g)$ is registered as A007323 at OEIS (on-line encyclopedia of integer sequences) and its first few elements are $1, 1, 2, 4, 7, 12, 23, 39, 67, 118$. Bras-Amorós \cite{Amoros1} conjectured the following items about the sequence $(n_g)$:
\begin{enumerate}
  \item $\lim_{g \to \infty} \frac{n_{g+1}}{n_g} = \varphi = \frac{1+\sqrt{5}}{2}$ (golden ratio);  
  \item $\lim_{g \to \infty} \frac{n_{g+1} + n_g}{n_{g+2}} = 1$;
  \item $n_g + n_{g+1} \leq n_{g+2}$, for all $g$.
\end{enumerate}

Zhai \cite{Zhai} gave a proof for the first two items; the key ingredient is that if $n'_g$ is the number of numerical semigroups of genus $g$ and depth at most 3, then $\lim_{g \to \infty} \frac{n'_{g}}{n_g} = 1$, which was a conjecture proposed by Zhao \cite{Zhao}. The third item remains as an open problem, as well as deciding if the sequence $(n_g)$ is non-decreasing.

Bras-Amorós \cite{Amoros} proved that $2\F_g \leq n_g \leq 1+3 \cdot 2^{g-3}$, for all $g \geq 3$, where $(\F_n)$ is the Fibonacci sequence given by $\F_0 = 0, \F_1 = 1$ and $\F_n = \F_{n-1} + \F_{n-2}$, for all $n \geq 2$. We observe that the lower bound is interesting because $2\F_g$ and $n_g$ are asymptotic to $c_1 \varphi^g$ and $c_2 \varphi^g$, respectively, where $\varphi$ is the golden ratio and $c_1$ and $c_2$ are constants. Some authors studied the numbers $n_g$ by taking partitions of set  of numerical semigroups of genus $g$, where a second invariant was considered (see \cite{GM, MF, BR, Amoros3}). In particular, some authors considered the multiplicity of numerical semigroups as the second invariant. Kaplan \cite{Kaplan} obtained important results and the Apéry set and the Kunz coordinates play important role in the constructions. García-Sánchez \textit{et al.} \cite{RPA} provided formulas for the number of numerical semigroups with fixed genus and multiplicities $3, 4$ and $5$; in particular, they partially verified a conjecture proposed by Kaplan \cite{Kaplan}. Karakas \cite{Karakas} has also used the Kunz coordinates of a numerical semigroup to parametrize the numerical semigroups of genus $g$ and multiplicity at most $5$. Eliahou and Fromentin \cite{EF2} used the so-called gapset filtrations to construct injective maps from the set of numerical semigroups of genus $g$ and multiplicity $m$ to the set of numerical semigroups of genus $g+1$ and multiplicity $m$, where $m = 3$ or $m = 4$.

It is important to observe that the problem of counting numerical semigroups by genus can be handled looking, exclusively, at the set of gaps of numerical semigroups; thereby, the concept of \textit{gapset} was formally introduced by Eliahou and Fromentin \cite{EF} and, in brief, it is the set of gaps of some numerical semigroup. However, it can be treated independently. A \textit{gapset} is a finite set $G \subset \N$ satisfying the following property: let $z \in G$ and write $z = x + y$, with $x$ and $y \in \N$; then $x \in G$ or $y \in G$. Naturally, all the terminology used for numerical semigroups can be transferred for gapsets (genus, multiplicity, conductor, depth, etc.). The main result of their paper is that $n'_{g} + n'_{g+1} \leq n'_{g+2} \leq n'_{g-1} + n'_{g} + n'_{g+1}$ holds true, for all $g \in \N$, which is remarkable, specially due to Zhai's result (cited above). In the same paper, Eliahou and Fromentin also introduced the notion of $m$-extension, for $m \in \N$, $m>1$. An $m$\textit{-extension} $A \subset \N$ is a finite set containing $[1,m-1]$ that admits a partition $A = A_0 \cup A_1 \cup \ldots \cup A_t$, for some $t \in \N_0$, where $A_0 = [1,m-1]$ and $A_{i+1} \subseteq m + A_i$ for all $i$. In particular, if $A$ is an $m$-extension, then $A \cap m\N = \emptyset$. Moreover, every gapset with multiplicity $m$ is an $m$-extension, but the converse does not hold true. We define the \textit{genus}, the \textit{conductor} and the \textit{depth} of an $m$-extension $A$ as $g(A) := \# A$, $c(A) := \min\{s \in \N_0: s + n \notin A, \forall n \in \N_0\}$ and $q(A) := \left\lceil\frac{c(A)}{m}\right\rceil$, respectively. In particular, one can prove that the number $t$ in the definition of an $m$-extension coincides with $q - 1$, where $q$ is its depth.

Recently, Bacher \cite{Bacher} considered a new approach to deal with numerical semigroups of genus $g$ and depth $q$ and Zhu \cite{Zhu} studied numerical semigroups of genus $g$ and depth at most $3$, by dealing with their Kunz coordinates.

Here is an outline of this paper. In section 2, we present some general properties of gapsets and of the $k$-generalized Fibonacci sequences. In section 3, we bring the terminology of Apéry set and Kunz coordinates to gapsets and we generalize those notions to $m$-extensions; moreover, for given $m \in \N, m > 1$, we construct a bijective map between the set of $m$-extensions and $\N^{m-1}$. In section 4, we introduce a bijective map $\sigma$ between the set of all $m$-extensions of genus $g$ and the set of tilings of a $g$-board, which allows us to identify gapsets with tiling of boards; in particular we prove the Bras-Amorós conjecture for $m$-extensions. In section 5, we present a lower bound for the number $n'_g$ which depends on the Fibonacci and Padovan sequences. In section 6, we provide a sequence of upper bounds for $n_g$, which depend on the $k$-generalized Fibonacci sequences and the number of gapsets with fixed genus, depth and multiplicity. In section 7, we compute the number of gapsets of genus $g$, depth $q$ and multiplicity 3 or 4, for all values of $g$ and $q$; moreover, we compute the number of gapsets with given genus and depth for some specific values.

\section{Auxiliary Results}

In this section, we present some well known results about gapsets, Fibonacci sequence, $k$-generalized Fibonacci sequences and Padovan sequence as well as some constructions that will be important in this paper. The first result is classical in gapset theory and it relates the genus of gapset with its depth.

\begin{proposition}
If a non-empty gapset has genus $g$ and depth $q$, then $1 \leq q \leq g$.
\label{dep}
\end{proposition}

\begin{proof}
Recall that if $G$ is a non-empty gapset of genus $g$, multiplicity $m$ and conductor $c$, then $2 \leq m \leq g+1$ and $g+1 \leq c \leq 2g$. Combining those inequalities, we obtain
$$1 = \frac{g+1}{g+1} \leq \frac{c}{m} \leq \frac{2g}{2} = g.$$
Hence, $1 \leq q \leq g$.
\end{proof}

Here, we point out that the \textit{ordinary gapset} $[1, g]$ attains the lower bound, the \textit{hyperelliptic gapset} $(2\N + 1) \cap [1,2g-1]$ attains the upper bound and the only gapset with depth $0$ is the empty gapset. A natural question that arises is the following: for each $k \in [1,g]$, is there a gapset of genus $g$ and depth $k$? We deal with this question in section 6.

The next result guarantees that every $m$-extension with depth at most 2 is a gapset.

\begin{proposition}
Let $G \subseteq [1, 2m-1]$, with $[1,m-1] \subseteq G$ and $m \notin G$. Then $G$ is a gapset with multiplicity $m$ and depth at most $2$.
\label{q2}
\end{proposition}

\begin{proof}
Let $z \in G$ and write $z = x + y$, with $x \leq y$. Since $z \leq 2m-1$, then $x \leq m-1$ and we conclude that $x \in G$; hence $G$ is a gapset. Moreover, the conductor of $G$ is at most $2m - 1$; thus, $q(G) \leq 2$.
\end{proof}

Recall that the Fibonacci sequence $(\F_n)$ is given by the recurrence relation \linebreak $\F_n = \F_{n-1} + \F_{n-2}$, for all $n \geq 2$ with initial values $\F_0 = 0$ and $\F_1 = 1$ (registered as A000045 at OEIS).

There are some generalizations of this well known sequence and in this paper we deal with the $k$-generalized Fibonacci sequences, namely $(\F^{(k)}_n)$, where $k \geq 2$ and $n \geq -k+2$. The recurrence relation is given by $\F^{(k)}_n = \sum_{i = 1}^k \F^{(k)}_{n-i}$ and the initial values are given by $\F^{(k)}_1 = 1$ and $\F^{(k)}_i = 0$, for $i \in [-k+2, 0]$. There is a formula for the first few elements of a $k$-generalized Fibonacci sequence and we present it as follows.

\begin{lemma}
Let $k \geq 2$ be an integer. Then $\F^{(k)}_{n} = 2^{n-2}$ for all $n \in [2,k+1]$.
\label{inicFibo}
\end{lemma}

\begin{proof}
We proceed by induction on $n$. Notice that $\F^{(k)}_{2} = \F^{(k)}_{1} = 1 = 2^0$. Let $n \in [2, k+1]$ and assume that $\F^{(k)}_{j} = 2^{j-2}$, for all $j \in [2, n-1]$. Then 
$$\F^{(k)}_{n} = \sum_{j=1}^{k} \F^{(k)}_{n-j} = 1 + \sum_{j=1}^{n-2} 2^{n-j-2} = 2^{n-2}$$ 
and we are done.
\end{proof}

If $k$ is an integer such that $k \geq 2$, the $k$-generalized Fibonacci sequence can be defined in an alternative way and we summarize it as follows. For a positive integer $g$, a $g$-composition is a sequence $(\b_1, \b_2, \ldots, \b_n)$, where $\b_i$ is a positive integer for all $i$ and $\sum \b_i = g$. A visual interpretation for a $g$-composition is considering a $g$-board (board $1 \times g$ with 1 row and $g$ columns) and associating an element $\b_i$ of that sequence with a rectangle $1 \times \b_i$; thus we can tile the $g$-board with the small rectangles $1 \times \b_1, 1 \times \b_2, \ldots, 1 \times \b_n$ without overlapping. Throughout this paper, we choose  to consider this visual interpretation and we denote the set of tilings of a $g$-board by $\mathcal{C}(g)$ and the set of tilings of a $g$-board with parts belonging to $\{1 \times 1, 1 \times 2, \ldots, 1 \times k\}$ by $\mathcal{C}(g, k)$. An element of $\mathcal{C}(g)$ will be denoted by its associated $g$-composition. 

\begin{example}
The tiling $1 \times 4, 1 \times 1$ of a $5$-board is denoted by $(4,1)$ and Figure \ref{tiling} gives a visual interpretation of it.
\begin{center}
\begin{figure}
\includegraphics[scale=0.8]{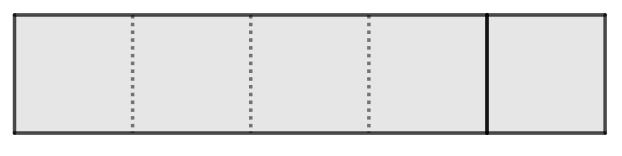}
\caption{The tiling $(4,1)$ of a $5$-board}
\label{tiling}
\end{figure}
\end{center}
\end{example}

A well known result is that the $k$-generalized Fibonacci number of order $n$ is given by

\begin{equation}
\F^{(k)}_{g+1} = \# \mathcal{C}(g, k),
\label{Cgk}
\end{equation}
that can be proved by induction on $g$, when $k$ is fixed.

As a consequence of Lemma \ref{inicFibo}, we have that $\F^{(g)}_{g+1} = 2^{g-1}$ and we obtain a formula for the number of tilings of a $g$-board.

\begin{proposition}
We have that $\# \mathcal{C}(g) = 2^{g-1}$.
\label{Cg}
\end{proposition}

\begin{proof}
Let $C$ be a tiling of a $g$-board. Then the parts of $C$ belong to $\{1 \times 1, 1 \times 2, \ldots, 1 \times g\}$. From (\ref{Cgk}), $\# \mathcal{C}(g) = \F^{(g)}_{g+1}$ and the formula provided in Lemma \ref{inicFibo} concludes the proof.
\end{proof}

Another sequence that is important in this paper is the so-called Padovan sequence. It can be defined as the sequence $(\PP_n)$ such that $\PP_n = \PP_{n-2} + \PP_{n-3}$ for all $n \geq -3$ with initial values $\PP_{-3} = 1$ and $\PP_{-2} = \PP_{-1} = 0$. The first few elements of this sequence are $1, 0, 0, 1, 0, 1, 1, 1, 2, 2, 3, 4, 5, 7, 9, 12, 16, \ldots$ and it is registered as A000931 in OEIS. An interpretation for the number $\PP_n$, when $n$ is a positive integer, is the number of $n$-compositions with parts $2$ or $3$. For instance, $\PP_7 = 3$, since there are three $7$-compositions with parts $2$ and $3$ (3 + 2 + 2, 2 + 3 + 2, 2 + 2 + 3). Notice that we can also give the visual interpretation for the number $\PP_7$ as the correspondent tiling of a $7$-board. A useful relation between the Fibonacci sequence and the Padovan sequence for this paper is the following.

\begin{proposition}
Let $(\F_n)$ and $(\PP_n)$ be the Fibonacci and the Padovan sequences, respectively. Then

$$\displaystyle \sum_{n = -3}^{g-3} \PP_n \cdot \F_{g - 2 - n} = \F_{g+2} - \PP_{g+1}, \forall g \in \N_0.$$
\label{Padovan}
\end{proposition}

\begin{proof}
We proceed by induction on $g$. If $g = 0$, then $\PP_{-3} \cdot \F_{1} = 1 = \F_{2} - \PP_{1}$. Let $g \in \N_0$ and suppose that the formula holds true for all $i \in [0, g]$. We shall compute $S = \sum_{n = -3}^{g-2} \PP_n \cdot \F_{g - 1 - n}$. Since $\F_{g - 1 - n} = \F_{g-2-n} + \F_{g-3-n}$ for all $n \in [-3, g-3]$, then 

\begin{eqnarray*}
S & = & \PP_{g-2} \cdot \F_{1} + \sum_{n = -3}^{g-3} \PP_n \cdot \F_{g - 2 - n} + \sum_{n = -3}^{g-4} \PP_n \cdot \F_{g - 3 - n} + \PP_{g-3} \cdot \F_0 \\
  & = & \PP_{g-2} + (\F_{g+2} - \PP_{g+1}) + (\F_{g+1} - \PP_{g}) \\
  & = & (\F_{g+2} + \F_{g+1}) - (\PP_{g} + \underbrace{\PP_{g+1} - \PP_{g-2}}_{\PP_{g-1}}) = \F_{g+3} - \PP_{g+2}
\end{eqnarray*}
and the result follows.
\end{proof}

\section{Generalization of the Ap\'{e}ry set and the Kunz coordinates for $m$-extensions}

In this section, we deal with the Ap\'{e}ry set and the Kunz coordinates of a gapset $G$. This definition arises in a natural way, when we consider the numerical semigroup $\N_0 \setminus G$. Moreover, we generalize the concepts of Ap\'{e}ry set and the Kunz coordinates for $m$-extensions.

Let $G$ be a gapset with multiplicity $m$ and consider the numerical semigroup \linebreak $S := \N_0 \setminus G$. The \textit{Ap\'{e}ry set} of $S$ (on $m$) is defined as $\Ap(S) = \{w_0, w_1, \ldots, w_{m-1}\}$, where \linebreak $w_i = \min\{s \in S: s \equiv i \pmod{m}\}$, for $i \in [0, m-1]$. Notice that $w_0 = 0$ and, for each $i$, there exits a unique $k_i \in \N$ such that $w_i = mk_i + i$. The \textit{Kunz coordinates} of $S$ (on $m$) are given by $\Kunz(S) = (k_1, k_2, \ldots, k_{m-1})$. Thus, it is possible to consider the Ap\'{e}ry set of $G$ (on $m$) and the Kunz coordinates of $G$ (on $m$) given by $\Ap(G) := \Ap(\N_0 \setminus G)$ and $\Kunz(G) := \Kunz(\N_0 \setminus G)$, respectively. Notice that the elements $w_i$ can be defined as $w_i := m + \max\{g \in G: g \equiv i \pmod{m}\}$, if $i \neq 0$ and $w_0 = 0$. 

We can extend those concepts and define the pseudo Ap\'{e}ry set and the pseudo Kunz coordinates of an $m$-extension as follows. Let $m > 1$ be an integer and $A$ be an $m$-extension. The \textit{pseudo Ap\'{e}ry set} of $A$ (on $m$) is defined as $\Ap^p(A) = \{w_0, w_1, \ldots, w_{m-1}\}$, where $w_i := m + \max\{a \in A: a \equiv i \pmod{m}\}$, if $i \neq 0$ and $w_0 = 0$. Notice that $w_i$ is well-defined for all $i$, since $A$ is a finite subset of $\N$ that can be written as $A = A_0 \cup A_1 \cup \cdots \cup A_{q-1}$, where $A_0 = [1, m-1]$ and $A_{i+1} \subseteq A_i + m$. By writing $w_i = mk_i + i$, for $i \in [1, m-1]$, the \textit{pseudo Kunz coordinates} of $A$ (on $m$) are given by $\Kunz^p(A) = (k_1, k_2, \ldots, k_{m-1})$. We illustrate this construction in the next example.

\begin{example}
Let $A = \{1, 2, 4, 7, 10\}$. It is a $3$-extension, but it is not a gapset \linebreak ($10 = 5 + 5$). In this case, $\Ap^p(A) = \{0, 13, 5\}$ and $\Kunz^p(A) = (4,1)$.
\label{3ext}
\end{example}

If $A$ is an $m$-extension and also a gapset, then the notions of pseudo Ap\'{e}ry set of $A$ and Ap\'{e}ry set of $A$ coincide, as well as the notions of pseudo Kunz coordinates of $A$ and Kunz coordinates of $A$, i.e., $\Ap^p(A) = \Ap(A)$ and $\Kunz^p(A) = \Kunz(A)$. 

The next result characterizes an $m$-extension from its pseudo Kunz coordinates.

\begin{proposition}
Let $A$ be an $m$-extension, with $\Kunz^p(A) = (k_1, k_2, \ldots, k_{m-1})$. If $i \in [1, m-1]$, then $A \cap \{n \in \N: n \equiv i \pmod{m}\} = \{i, i + m, \ldots, i + (k_i - 1)m\}$. Moreover,  $k_i = \#\{a \in A: a \equiv i \pmod{m}\}$, for all $i \in [1, m-1]$.
\label{Kunz}
\end{proposition}

\begin{proof}
By the definition of $\Kunz^p(A)$, we have that $m(k_i - 1) + i \in A$, for all $i$ and $mx + i \notin A$, if $x \geq k_i$. Since $A_{j+1} \subseteq m + A_j$ for all $j$, then it follows that \linebreak $m(k_i - \l) + i \in A$, for all $\l \in [1, k_i]$. Thus, the elements of $A$ that are congruent to $i$ modulo $m$ are $i, i + m, \ldots, i + (k_i - 1)m + i$, totalizing $k_i$ elements.
\end{proof}

As a consequence we obtain formulas for the genus and the depth of an $m$-extension in terms of its pseudo Kunz coordinates.

\begin{corollary}
Let $A$ be an $m$-extension with $\Kunz^p(A) = (k_1, k_2, \ldots, k_{m-1})$, genus $g$ and depth $q$. Then $g = \sum_{i=1}^{m-1} k_i$ and $q = \max\{k_i: 1 \leq i \leq m-1\}$.
\label{genus_depth_kunz}
\end{corollary}

\begin{proof}
By Proposition \ref{Kunz}, $A \cap \{n \in \N: n \equiv i \pmod{m}\}$ has $k_i$ elements and the formula for the the genus can be obtained by summing up the coordinates of $\Kunz^p(A)$. Write $A = A_0 \cup A_1 \cup \cdots \cup A_{q-1}$, with $A_0 = [1, m-1]$ and $A_{i+1} \subseteq A_i + m$ and let $x \in A_{q-1}$. For each $j \in [0, q-1]$, there is exactly one element in $A_j$ that is congruent to $x \pmod{m}$. By Proposition \ref{Kunz}, we conclude that $q = k_{t}$, where $t$ coincides with the remainder when $x$ is divided by $m$ and we are done.
\end{proof}

\begin{example}
The $3$-extension such that $\Kunz^p(A) = (4,1)$ is $A = \{1, 2, 4, 7, 10\}$. Its genus is $5$, its conductor is $11$ and its depth is $4 = \left \lceil \frac{11}{3} \right \rceil$. The remainder when 10 is divided by 3 is 1 and the depth of $A$ is exactly the first coordinate of $\Kunz^p(A)$. The $4$-extension such that $\Kunz^p(G) = (4,4,3)$ is $G = \{1, 2, 3, 5, 6, 7, 9, 10, 11, 13, 14\}$. Its genus is $11$, its conductor is $15$ and its depth is $4 = \left \lceil \frac{15}{4} \right \rceil$. The remainder when 13 is divided by 4 is 1 and the depth of $A$ is exactly the first coordinate of $\Kunz^p(G)$ (this procedure could also be done with the number 14).
\label{ex34}
\end{example}

We recall that if $S$ is a numerical semigroup with multiplicity $m$, then its Kunz coordinates (on $m$), $\Kunz(S) = (k_1, k_2, \ldots, k_{m-1})$, must satisfy the following system (cf. \cite{Rosales}):

\begin{equation}
\begin{cases}
X_i \in \N \\
X_i + X_j \geq  X_{i+j}, \hspace{1.2cm} \text{ for } 1 \leq i \leq j \leq m-1; i + j < m; \\
X_i + X_j + 1 \geq X_{i+j-m}, \text{ for } 1 \leq i \leq j \leq m-1; i + j > m
\end{cases}
\label{system}
\end{equation}

Hence, one can characterize the Kunz coordinates of a gapset in the same way.

\begin{remark}
It is also possible to prove Proposition \ref{q2} using the fact that the pseudo Kunz coordinates of an $m$-extension with depth at most $2$ satisfies the system (\ref{system}).
\end{remark}

One can show that there is bijective map between the set of numerical semigroups with multiplicity $m$ and the set of points in $\N^{m-1}$ that satisfy the system (\ref{system}) (via the Kunz coordinates) (see \cite{Rosales}). 

\begin{example}
The tuple $(4,4,3)$ satisfies the system (\ref{system}). Hence, there is a gapset $G$ such that $\Kunz(G) = (4,4,3)$. Following the Example \ref{ex34}, we obtain \linebreak $G = \{1,2,3,5,6,7,9,10,11,13,14\}$.
\label{gapsetKunz}
\end{example}

If $\mathcal{A}(m)$ is the set of $m$-extensions, then the map $\mu: \mathcal{A}(m) \to \N^{m-1}, A \mapsto \Kunz^p(A)$ is injective. The next result ensures that $\mu$ is a bijective map.

%PROVA INJETIVIDADE%
%Let $A$ be an $m$-extension and $A'$ be an $m'$-extension such that $(k_1, k_2, \ldots, k_{m-1}) = \Kunz^p(A) = \Kunz^p(A') = (k'_1, k'_2, \ldots, k'_{m'-1})$. Hence, $m = m'$ and $k_i = k'_i$, for all $i$. Thus, $A$ and $A'$ have exactly the same elements and $A = A'$.

\begin{theorem}
The map $\mu$ is surjective.
%Every $(m-1)$-tuple of positive integers is given as the pseudo Kunz coordinates of some $m$-extension.
\label{surjective}
\end{theorem}

\begin{proof}
Let $(k_1, k_2, \ldots, k_{m-1}) \in \N^{m-1}$ and consider the finite set of $\N$ given by 

$$A = \bigcup_{j=1}^{m-1} \{j, j+m, \ldots, j + (k_j-1)m\}.$$

Now we prove that $A$ is an $m$-extension. Write $A = A_0 \cup A_1 \cup \ldots \cup A_{q-1}$, where \linebreak $q = \max\{k_j: 1 \leq j \leq m-1\}$ and $A_i = [im + 1, (i+1)m - 1] \cap A$. If $x \in A$, then $x \in A_i$ for some $i \in [0, q]$ and $x \equiv j \pmod{m}$ for some $j \in [1, m-1]$. If $i = 0$, then there is nothing to show. If $i > 0$, then $x - m \in [(i-1)m + 1, im - 1]$ and $x - m \in A$ (by construction). Thus, $x - m \in A_{i-1}$. Hence, $A_i \subseteq A_{i-1} + m$, for all $i \in [1, q]$ and $A$ is an $m$-extension. Moreover, $\mu(A) = (k_1, k_2, \ldots, k_{m-1})$.
\end{proof}

\begin{example}
The tuple $(1,3,3,2)$ is not associated to any gapset, because it does not satisfy the system (\ref{system}) ($k_1 + k_1 = 2 < 3 = k_2$). However, Proposition \ref{surjective} ensures that it is associated to a $5$-extension, namely $A = \{1,2,3,4,7,8,9,12,13\}$.
\label{naoexemplo}
\end{example}

\section{Identifying $m$-extensions with tiling of boards}
\label{ident}

In this section, we are interested in identifying an $m$-extension of genus $g$ with a tiling of a $g$-board, which is a board $1 \times g$ ($1$ row and $g$ columns), where $g$ is a positive integer. In particular, it allows us to identify a gapset of genus $g$ with a tiling of a $g$-board. Throughout, we denote the set of all $m$-extensions of genus $g$ by $\mathcal{A}_g$ and the set of tilings of a $g$-board by $\mathcal{C}(g)$.

Let $A$ be an $m$-extension with $\Kunz^p(A) = (k_1, k_2, \ldots, k_{m-1})$. Recall that the genus and the depth of $A$ are given by $\sum_{i=1}^{m-1} k_i$ and \linebreak $\max\{k_i: 1 \leq i \leq m-1\}$, respectively. We identify the $m$-extension $A$ with the tiling \linebreak $(k_1, k_2, \ldots, k_{m-1})$, i.e., the first part of the tiling has size $1 \times k_1$, the second part of the tiling has size $1 \times k_2$ and so on, up to the $(m-1)$-th part of the tiling that has size $1 \times k_{m-1}$. We observe that some of the invariants can be easily deduced, when we identify an $m$-extension of genus $g$ with a tiling of a $g$-board, via its pseudo Kunz coordinates:

\begin{itemize}
  \item the genus of the $m$-extension is the size of the associated board;
  \item the number $m$ of the $m$-extension is the number of parts of the tiling plus one;
  \item the depth of the $m$-extension is the size of the greatest part of the tiling.
\end{itemize}

We illustrate with the following example.

\begin{example}
Let $A = \{1,2,4,7,10\}$ be the $3$-extension of the Example \ref{3ext}, which has genus $5$, depth $4$ and $\Kunz^p(A) = (4, 1)$. %In this case, $A$ is identified with the following element of $\mathcal{C}(5)$:

%\begin{center}
%\includegraphics[scale=0.8]{extension.png}
%\end{center}

\end{example}

Now we prove that the association of an $m$-extension with its pseudo Kunz coordinates induces a bijection between the set of all $m$-extensions of genus $g$ and the set of tilings of a $g$-board. 

\begin{theorem}
Let $g$ be a positive integer. Then $\sigma: \mathcal{A}_g \to \mathcal{C}(g), A \mapsto \sigma(A) = \Kunz^p(A)$ is a bijective map.
\end{theorem}

\begin{proof}
Clearly, the map $\sigma$ is well defined and injective. Now we prove that $\sigma$ is surjective. If $(c_1, c_2, \ldots, c_n) \in \mathcal{C}(g)$, then $\sum_{i=1}^{n} c_i = g$.  Following the construction of the proof of Theorem \ref{surjective}, we shall consider the $(n+1)$-extension given by 

$$A = \bigcup_{j=1}^{n} \{j, j+(n+1), \ldots, j + (c_j-1)(n+1)\}.$$

Hence, $A$ has genus $g$ and $\sigma(A) = (c_1, c_2, \ldots, c_n)$.

%We shall construct an $m$-extension $A$ such that $\Kunz(A) = (c_1, c_2, \ldots, c_n)$ for some $m$. We must take $m = n + 1$ and $A = \bigcup_{i = 1}^{m-1} \{i, m + i, \ldots, (c_i-1)m + i\}$. Thus, this construction ensures that $A$ is an $m$-extension.
\end{proof}

\begin{corollary}
The cardinality of $\mathcal{A}_g$ is exactly $2^{g-1}$ for all positive integer $g$. %Moreover, $\#\mathcal{A}(0) = 1$.
\label{numbermext}
\end{corollary}

\begin{proof}
The bijection $\sigma$ guarantees that the cardinality of $\mathcal{A}_g$ coincides with the cardinality of $\mathcal{C}(g)$, which is $2^{g-1}$ by Proposition \ref{Cg}. %Moreover, $\emptyset$ is the only $m$-extension of genus $0$.
\end{proof}

\begin{remark}
Let $(a_g)$ be the sequence of the number of $m$-extensions of genus $g$. As a consequence of Corollary \ref{numbermext}, we conclude that
\begin{itemize}
  \item $a_{g} + a_{g+1} = \frac{3}{4}a_{g+2} < a_{g+2}$, for all positive integer $g$
  \item $\frac{a_{g+1}}{a_{g}} = 2$, for all positive integer $g$
  \item $\frac{a_{g+1} + a_{g}}{a_{g+2}} = \frac{3}{4}$, for all positive integer $g$
\end{itemize}
which is a proved version of Bras-Amorós' conjecture for $m$-extensions.
\label{BA_mext}
\end{remark}

As a consequence of the definition of the map $\sigma$, the image of an $m$-extension of genus $g$ and depth $q$ is a tiling in $\mathcal{C}(g,k)$, for all $k \geq q$; reciprocally, if $C$ is a tiling in $\mathcal{C}(g,q)$, then $\sigma^{-1}(C)$ is an $m$-extension of genus $g$ and depth at most $q$. Let $\Gamma(g)$ be the set of gapsets of genus $g$ and $\Gamma(g,q) := \{G \in \Gamma(g): q(G) \leq q\}$. It was proved by Zhao \cite{Zhao} that $\#\Gamma(g,2) = \F_{g+1}$ by using an approach involving numerical semigroups and an identity involving the Fibonacci numbers and binomial coefficients. Eliahou and Fromentin \cite{EF} proved that the sequences $(\#\Gamma(g,2))$ and $(\F_{g+1})$ satisfies the same recurrence relation and have the same initial conditions. Here we give a third proof for this fact.

%Now we deal with the map $\sigma$ restricted to $\Gamma(g,q)$ and we present two results. The first one involves the number of gapsets of genus $g$ and depth at most $2$. 

\begin{proposition}
The number of gapsets of genus $g$ and $q \leq 2$ is $\F_{g+1}$.
\label{q=2}
\end{proposition}

\begin{proof}
Consider the map $\sigma$ restricted to $\Gamma(g,2)$. We only have to determine the image of this application. We claim that it coincides with the set $\mathcal{C}(g,2)$. In fact, every gapset of genus $g$ and depth at most $2$ is identified with a tiling in $\mathcal{C}(g,2)$. Moreover, if $C$ is a tiling in $\mathcal{C}(g,2)$, then it represents an $m$-extension with depth at most $2$. By Proposition \ref{q2}, this $m$-extension is a gapset. Hence, the preimage of $C$ is a gapset of genus $g$ and depth at most 2. Hence, there is a bijection between $\Gamma(g,2)$ and $\mathcal{C}(g,2)$ and the result follows.
\end{proof}

We also obtain an upper bound for the number of gapsets of genus $g$ and depth at most $k$ depending on the $k$-generalized Fibonacci numbers.

\begin{proposition}
The number of gapsets of genus $g$ and depth at most $k$ is bounded above by $\F^{(k)}_{g+1}$.
\label{cotasup}
\end{proposition}

\begin{proof}
Every gapset of genus $g$ and depth at most $k$ has image under the map $\sigma$ being a tiling in $\mathcal{C}(g,k)$. Hence, the map $\sigma$ restricted to $\Gamma(g, k)$ has as image a subset of $\mathcal{C}(g,k)$. From (\ref{Cgk}), the result follows.
\end{proof}

In particular, if $(\T_{g})$ denotes the Tribonacci sequence, then $n'_g \leq \T_{g+1}$
. It has been observed by Eliahou and Fromentin \cite{EF} as a consequence of the following remarkable relation involving the sequence $(n'_ g)$: $n'_{g+3} \leq n'_{g} + n'_{g+1} + n'_{g+2}$, for all $g$ and $(n'_0, n'_1, n'_2) = (1, 1, 2) = (\T_1, \T_2, \T_3)$. Another immediate consequence of Proposition \ref{cotasup} is the following upper bound for $n_g$.

\begin{corollary}
Let $n_g$ be the number of gapsets of genus $g > 0$. Then  $n_g \leq 2^{g-1}$.
\label{cotan_g}
\end{corollary}

\begin{proof}
The number $n_g$ can be defined as $n_g = \{G \in \Gamma(g): q(G) \leq g\}$. Hence, it coincides with the cardinality of $\#\Gamma(g, g)$. From Proposition \ref{cotasup}, we obtain the result.
\end{proof}

\section{A lower bound for $n'_g$}

In this section, we provide a lower bound for the number $n'_g$, which is obtained using the identification given in section 4. For this propose, we use the fact that there exist exactly $\F_{g+1}$ gapsets of genus $g$ and depth at most 2 (cf. Proposition \ref{q=2}) and we construct a family of gapsets of genus $g$, multiplicity $m$ and depth $3$. The next result exhibits such a family.

\begin{theorem}
Let $a$ be a positive integer and $A$ be an $m$-extension such that \linebreak $\Kunz^p(A) = (k_1, \ldots k_{a-1}, k_a, k_{a+1} \ldots, k_{m-1})$. If $k_\l \in \{2, 3\}$ for all $\l \in [1, a-1], k_a = 3$ and $k_\l \in \{1, 2\}$, for all $\l \in [a+1, m-1]$, then $A$ is a gapset. Moreover, $A$ has depth 3.
\label{gapsetsq=3}
\end{theorem}

\begin{proof}
We prove that $\Kunz^p(A)$ satisfies the system (\ref{system}). Let $i$ and $j$ be indexes lying in $[1, m-1]$. If $i$ and $j$ are such that $i + j > m$, then, $k_i + k_j + 1 \geq 3 \geq k_{i+j - m}$. Now, let $i$ and $j$ be such that $i + j < m$. If $i$ and $j \in [1, a]$, then $k_i + k_j \geq 4 > k_{i + j}$. If $i \in [1, a]$ and $j \in [a+1, m-1]$, then $k_i + k_j \geq 3 \geq k_{i + j}$. If $i$ and $j \in [a+1, m-1]$, then $k_i + k_j \geq 2 \geq k_{i + j}$ and we are done.
\end{proof}

\begin{remark}
Another way to prove Theorem \ref{gapsetsq=3} is proving that the associated $m$-extension $A$ satisfies the definition of a gapset.
\end{remark}

\begin{proposition}
Let $(\F_n)$ and $(\PP_n)$ be the Fibonacci and the Padovan sequences, respectively and consider a $g$-board divided in three parts. The quantity of tilings of a $g$-board, where the first part has tiles of types $1 \times 2$ and $1 \times 3$ (in any order), the second part has a unique tile of type $1 \times 3$ and the last part has tiles of types $1 \times 1$ or $1 \times 2$ (in any order) is given by

$$\F_{g-2} + \sum_{n = 2}^{g-3}\PP_n \cdot \F_{g - 2 - n}.$$

\label{fibopado}
\end{proposition}

\begin{proof}
Let $n$ be the size of the first part. If $n = 0$, then the second part is $1 \times 3$ and there are $\F_{g-2}$ tilings of a $(g-3)$-board with parts $1 \times 1$ and $1 \times 2$. If $n$ is a positive integer, then $n \geq 2$ and there are $\PP_n$ forms to tile the first part. The second part is of the type $1 \times 3$ and there are $\F_{g-2-n}$ forms to tile the final part, which can be seen as a $(g - n - 3)$-board. By running $n$ from 2 to $g - 3$, we obtain the result.
\end{proof}

\begin{corollary}
The number of gapsets of genus $g$ and depth $q \leq 3$ (and also the number gapsets of genus $g$) is bounded below by

$$\F_{g+2} - \PP_{g+1}.$$
\end{corollary}

\begin{proof}
By Proposition \ref{gapsetsq=3}, every tiling considered in Proposition \ref{fibopado} is associated to a gapset. Moreover, the number of gapsets of genus $g$ and depth at most 2 is $\F_{g+1}$. In particular, the number obtained in Proposition \ref{fibopado} plus $\F_{g+1}$ is a lower bound for $n'_g$. Hence, by the definition of the Padovan numbers with indexes $-3, -2, -1, 0$ and $1$ and Proposition \ref{Padovan}, we conclude that

$$n'_g \geq \sum_{n = -3}^{g-3}\PP_n \cdot \F_{g - 2 - n} = \F_{g+2} - \PP_{g+1}.$$
\end{proof}

As far as we know, this is the first explicit global lower bound for the number $n'_g$ that does not depend on the elements of the sequence itself. In \cite{EF}, the authors proves that $n'_g \geq n'_{g-1} + n'_{g-2}$ for $g \geq 2$ and $n'_g \geq 8\F_g$, for $g \geq 64$. Some authors provided lower bounds for $n_g$ (see \cite{Amoros, Elizalde}), with no reference to the depth of the gapset. We present the bound obtained by Bras-Amorós \cite{Amoros}, our bound, the bound obtained in \cite{EF} in Table \ref{tabela} and the precise values for $n'_g$ and $n_g$. We do not present the bounds obtained in \cite{Elizalde}, because he considers gapsets with depth greater than 3 in his construction.

\begin{table}[h]
\centering
\begin{tabular}{|c|c|c|c|c|c|}
  \hline
   $g$ & $2\F_g$ & $\F_{g+2} - \PP_{g+1}$ &  $n'_{g-1} + n'_{g-2}$ & $n'_g$ & $n_g$ \\
     \hline                               
     0 & * & 1& * & 1 & 1 \\
     \hline
     1 & * & 1 & * & 1 & 1 \\
     \hline
     2 & 2 & 2 & 2 & 2 & 2 \\
     \hline
     3 & 4 & 4 & 3 & 4 & 4 \\
     \hline
     4 & 6 & 6 & 6 & 6 & 7 \\
     \hline
     5 & 10 & 11 & 10 & 11 &  12 \\
     \hline
     6 & 16 & 18 & 17 & 20 & 23 \\
     \hline
     7 & 26 & 30 & 31 & 33 & 39 \\
     \hline
     8 & 42 & 50 & 53 & 57 & 67 \\
     \hline
     9 & 68 & 82 & 90 & 99 & 118 \\
     \hline
     10 & 110 & 135 & 156 & 168 & 204 \\
     \hline
%     11 & 178 & 221 & 267 & 287 & 343 \\
%     \hline
%     12 & 288 & 361 & 455 & 487 & 592 \\
%     \hline
%     13 & 466 & 589 & 774 & 824 & 1001 \\
%     \hline
%     14 & 754 & 959 & 1311 & 1395 & 1693 \\
%     \hline     
     \end{tabular}
   \vspace{0.3cm}
   \caption{Lower bounds for $n'_g$ and/or $n_g$ and their exact values}
   \label{tabela}
\end{table}

\section{A family of upper bound for $n_g$}

In this section, we prove an important relation between the genus, the depth and the multiplicity of a gapset. Moreover, we provide a family of upper bounds for the number $n_g$, which is another application of the identification of a gapset as a tiling of a $g$-board. As a consequence, we combine this identification with a result by Kaplan \cite{Kaplan} that ensures that the number of gapsets with multiplicity $m$ and genus $g$ agrees with a quasipolynomial in $g$ with degree $m - 2$ for large enough values of $g$. Throughout, the set of gapsets of genus $g$ and depth $q$ and the set of gapsets of genus $g$, depth $q$ and multiplicity $m$ are denoted by $\mathcal{F}(g,q)$ and $\mathcal{F}(g,q,m)$, respectively.
 
The first result gives an upper bound for the depth of a gapset in terms of its genus and its multiplicity.
 
 \begin{proposition}
Let $G$ be a gapset of genus $g$, depth $q$ and multiplicity $m$. Then $q \leq \left \lceil \frac{2g}{m} \right \rceil$.
\label{m}
\end{proposition}

\begin{proof}
Let $c$ be the conductor of $G$. Since $c \leq 2g$, we obtain $\frac{c}{m} \leq \frac{2g}{m}$ and the result follows.
\end{proof}

\begin{remark}
This bound is not sharp in general. For instance, if $m = 3$ and $g \equiv 2 \pmod{3}$, then there is not any gapset of genus $g$, multiplicity $3$ and depth $\left \lceil \frac{2g}{3} \right \rceil = \frac{2g + 2}{3}$ (see Theorem \ref{Fgq3}). On the other hand, if $m = 4$ and $g \geq 3$, then there is at least one gapset of genus $g$, multiplicity $4$ and depth $\left \lceil \frac{g}{2} \right \rceil$ (see Theorem \ref{Fgq4} and Table \ref{gq4}).
\end{remark}

We conclude that if $G$ is gapset of genus $g$, which is non-empty ($q \neq 0$), non-ordinary ($q \neq 1$) and non-hyperelliptic ($q \neq g$), then the multiplicity of $G$ is at least $3$. Hence, its depth lies in $[2, \lceil \frac{2g}{3} \rceil]$.  In particular, we have the following.

\begin{corollary}
If $g \geq 6$, then $\#\mathcal{F}(g,q) = 0$, for all $q \in \left[\left\lceil \frac{2g}{3} \right\rceil + 1, g-1 \right]$.
\label{zeros}
\end{corollary}

In particular, it improves the upper bound obtained in Proposition \ref{dep} in the following way:

\begin{proposition}
If a non-empty and non-hyperelliptic gapset has genus $g$ and depth $q$, then $1 \leq q \leq \left\lceil \frac{2g}{3} \right\rceil$.
\label{dep2}
\end{proposition}

By using a similar argument as in the proof of Proposition \ref{m}, we conclude that the genus $g$, the depth $q$ and the multiplicity $m$ of a gapset satisfies $q \geq \left \lceil \frac{g+1}{m} \right \rceil$. The next result improves this bound.

\begin{proposition}
Let $G$ be a non-empty gapset of genus $g$, depth $q$ and multiplicity $m$. Then $q \geq \left \lceil \frac{g}{m-1} \right \rceil$.
\label{mgqupper}
\end{proposition}

\begin{proof}
Let $(k_1, k_2, \ldots, k_{m-1})$ be the Kunz coordinates of $G$. From Corollary \ref{genus_depth_kunz}, \linebreak $g = \sum k_i \leq (m-1) \cdot q$ and the result follows.
\end{proof}

\begin{remark}
The bound presented in Proposition \ref{mgqupper} is sharp. In fact, let $m > 1$ and $g$ be positive integers such that $\left \lceil \frac{g}{m-1} \right \rceil = \frac{g+\epsilon}{m-1}$, where $\epsilon \in [0, m-2]$. Then the $m$-extension $G$ such that 
$$\Kunz(G) = \left(\underbrace{\frac{g+\epsilon}{m-1}, \frac{g+\epsilon}{m-1}, \ldots, \frac{g+\epsilon}{m-1}}_{m-1-\epsilon}, \underbrace{\frac{g+\epsilon}{m-1} - 1, \frac{g+\epsilon}{m-1} - 1, \ldots, \frac{g+\epsilon}{m-1} - 1}_{\epsilon}\right)$$
is a gapset of genus $g$, multiplicity $m$ and depth $\frac{g+\epsilon}{m-1}$.
\end{remark}

As a consequence of Propositions \ref{m} and \ref{mgqupper}, we obtain the following result.

\begin{corollary}
Let $G$ be a non-empty gapset of genus $g$, depth $q$ and multiplicity $m$. Then 
$$\left \lceil \frac{g}{m-1} \right \rceil \leq q \leq \left \lceil \frac{2g}{m} \right \rceil.$$
\label{mgq}
\end{corollary}
%
%The values in Table \ref{tab_mgq} show some examples of Corollary \ref{mgq}.
%
%\begin{table}[h]
%\centering
%\begin{tabular}{|c|c|c|}
%  \hline
%   $m$ & $q \geq$ & $q \leq$ \\
%     \hline                               
%     3 & $\left \lceil \frac{g}{2} \right \rceil$ & $\left \lceil \frac{2g}{3} \right \rceil$ \\
%     \hline
%     4 & $\left \lceil \frac{g}{3} \right \rceil$ & $\left \lceil \frac{g}{2} \right \rceil$ \\
%     \hline
%     5 & $\left \lceil \frac{g}{4} \right \rceil$ & $\left \lceil \frac{2g}{5} \right \rceil$ \\
%     \hline
%     \end{tabular}
%   \vspace{0.3cm}
%   \caption{Bounds for $q$ in terms of $m$ and $g$}
%   \label{tab_mgq}
%\end{table}

Next, we are interested in providing upper bounds for the number $n_g$. Let $M$ be a positive integer. Then one can write

\begin{equation}
n_g = \sum_{q \leq \left \lceil \frac{2g}{M+1} \right \rceil} \#\mathcal{F}(g,q) + \sum_{q > \left \lceil \frac{2g}{M+1} \right \rceil} \#\mathcal{F}(g,q),
\label{ng}
\end{equation}
where $\mathcal{F}(g,q)$ denotes the set of gapsets of genus $g$ and depth $q$.

In particular, if $M = 1$, then (\ref{ng}) becomes $n_g =   \sum_{q \leq g} \#\mathcal{F}(g,q)$. Thus we can use Proposition \ref{cotasup} to obtain the statement of Corollary \ref{cotan_g}, i.e., $n_g \leq\F^{(g)}_{g+1} = 2^{g-1}$. 
Next, we improve this last upper bound. 

\begin{theorem}
Let $g$ and $M > 1$ be positive integers. Then
$$n_g \leq \F_{g+1}^{(\left \lceil \frac{2g}{M+1} \right \rceil)} + \sum_{m=2}^{M} \sum_{q = \left \lceil \frac{2g}{M+1} \right \rceil + 1}^{\left \lceil \frac{2g}{m} \right \rceil} \#\mathcal{F}(g,q,m).$$
\label{uppern_g}
\end{theorem}

\begin{proof}
Let $M > 1$ and $g$ be positive integers and denote by $\mathcal{F}(g,q,m)$ the set of gapsets of genus $g$, depth $q$ and multiplicity $m$. By Proposition \ref{cotasup}, the first sum of (\ref{ng}) is bounded above by $\F_{g+1}^{(\left \lceil \frac{2g}{M+1} \right \rceil)}$; by Corollary \ref{mgq}, if $G \in \mathcal{F}(g,q)$, with $q > \left \lceil \frac{2g}{M+1} \right \rceil$, then the multiplicity of $G$ is at most $M$. Thus, one can write

$$\sum_{q > \left \lceil \frac{2g}{M+1} \right \rceil} \#\mathcal{F}(g,q) = \sum_{q > \left \lceil \frac{2g}{M+1} \right \rceil} \sum_{m=2}^{M} \#\mathcal{F}(g,q,m).$$
By changing the order of the sum and using the fact that the depth $q$ of a gapset of genus $g$ and multiplicity $m$ is at most $\left \lceil \frac{2g}{m} \right \rceil$, we conclude that
$$\sum_{q > \left \lceil \frac{2g}{M+1} \right \rceil} \#\mathcal{F}(g,q) = \sum_{m=2}^{M} \sum_{q = \left \lceil \frac{2g}{M+1} \right \rceil + 1}^{\left \lceil \frac{2g}{m} \right \rceil} \#\mathcal{F}(g,q,m),$$

and the result follows.
\end{proof}

The next example illustrates how we can use the formula obtained in Theorem \ref{uppern_g} in the cases $M = 2, M = 3$ and $M = 4$.

\begin{example}
Recall that there is exactly one gapset of genus $g \geq 1$ and depth $q = g$ (the hyperelliptic gapset, which has multiplicity 2). The first values for $M$ give the following upper bounds for $n_g$, if $g \geq 1$.
\begin{itemize}
  \item[(a)] $(M = 2): n_g \leq \F_{g+1}^{(\left \lceil \frac{2g}{3} \right \rceil)} + \sum_{q = \left \lceil \frac{2g}{3} \right \rceil + 1}^{g} \# \mathcal{F}(g,q,2) = \F_{g+1}^{(\left \lceil \frac{2g}{3} \right \rceil)} + 1.$
  \item[(b)] $(M = 3): n_g \leq \F_{g+1}^{(\left \lceil \frac{2g}{4} \right \rceil)} + \sum_{q = \left \lceil \frac{2g}{4} \right \rceil + 1}^{\left \lceil \frac{2g}{3} \right \rceil} \# \mathcal{F}(g,q,3) + \sum_{q = \left \lceil \frac{2g}{4} \right \rceil + 1}^{g} \# \mathcal{F}(g,q,2)$, i.e., $$n_g \leq \F_{g+1}^{(\left \lceil \frac{g}{2} \right \rceil)} + \sum_{q = \left \lceil \frac{g}{2} \right \rceil + 1}^{\left \lceil \frac{2g}{3} \right \rceil} \# \mathcal{F}(g,q,3) +1.$$
  \item[(c)] $(M = 4):$ \\
{\small $n_g \leq \F_{g+1}^{(\left \lceil \frac{2g}{5} \right \rceil)} + \sum_{q = \left \lceil \frac{2g}{5} \right \rceil + 1}^{\left \lceil \frac{g}{2} \right \rceil} \# \mathcal{F}(g,q,4) + \sum_{q = \left \lceil \frac{2g}{5} \right \rceil + 1}^{\left \lceil \frac{2g}{3} \right \rceil} \# \mathcal{F}(g,q,3) + \sum_{q = \left \lceil \frac{2g}{5} \right \rceil + 1}^{g} \# \mathcal{F}(g,q,2)$,}\\ 
i.e., $$n_g \leq \F_{g+1}^{(\left \lceil \frac{2g}{5} \right \rceil)} + \sum_{q = \left \lceil \frac{2g}{5} \right \rceil + 1}^{\left \lceil \frac{g}{2} \right \rceil } \# \mathcal{F}(g,q,4) + \sum_{q = \left \lceil \frac{g}{2} \right \rceil}^{\left \lceil \frac{2g}{3} \right \rceil } \# \mathcal{F}(g,q,3) +1.$$
\end{itemize}
\label{upper_casos}
\end{example}

In order to obtain those upper bounds explicitly, we shall provide formulas for $\#\mathcal{F}(g,q,m)$. However, it seems to be a difficult problem to compute $\#\mathcal{F}(g,q,m)$ for arbitrary values of $g, q$ and $m$. In section 7, we present explicit formulas in the cases $m = 3$ and $m = 4$.

Here, we give an alternative to construct a family of upper bounds for $n_g$. We point out that the number of gapsets with multiplicity $m$ and genus $g$, namely $N(m,g)$, is an upper bound for $\sum_{q > \left \lceil \frac{2g}{M+1} \right \rceil} \#\mathcal{F}(g,q,m)$, which is a consequence of Corollary \ref{mgq}. In particular,

\begin{equation}
n_g \leq \F_{g+1}^{(\left \lceil \frac{2g}{M+1} \right \rceil)} + \sum_{m=2}^{M} N(m, g).
\label{upper}
\end{equation}

The next result appears as Proposition 7 of \cite{Kaplan} and it indicates how the numbers $N(m,g)$ look like, at least for $g$ large enough. Recall that a quasipolynomial is a function \linebreak $f: \N \to \C$ such that there exists a positive integer $t$ and polynomials $f_0, f_1, \ldots, f_{t-1}$ such that $f(n) = f_i(n)$, if $n \equiv i \pmod{t}$. The integer $t$ is the period of the quasipolynomial and the degree of the quasipolynomial is the largest degree of the polynomials $f_i$.

\begin{lemma}[\cite{Kaplan}]
Let $N(m,g)$ be the number of gapsets of genus $g$ and multiplicity $m$. If $m > 1$ is a positive integer, then $N(m, g)$ agrees with a quasipolynomial of degree $m - 2$ with a period depending on $m$ for $g \gg 0$.
\label{quasipol}
\end{lemma}

The next result is an immediate consequence of Lemma \ref{quasipol} and (\ref{upper}).

\begin{corollary}
\label{cor610} Let $M$ be a positive integer. Then there exists a quasipolynomial $p_M$ of degree $M-2$ such that

$$n_g \leq \F_{g+1}^{(\left \lceil \frac{2g}{M+1} \right \rceil)} + p_M(g),$$

for $g \gg 0$.
\end{corollary}

For instance, it is well known that $N(2, g) = 1$, for all $g \geq 1$, $N(3,g) = \left \lfloor \frac{g}{3}\right \rfloor + 1$, for all $g \geq 2$ (see \cite{B_GS_P}) and $N(4, g) = \left \lfloor \frac{g^2 + 6g}{12} \right \rfloor$, for all $g \geq 4$ (see \cite{Alhajjar}). From (\ref{upper}), with $M = 4$, we conclude that

$$n_g \leq \F_{g+1}^{(\left \lceil \frac{2g}{5} \right \rceil)} + \left \lfloor \frac{g^2 + 6g}{12} \right \rfloor + \left \lfloor \frac{g}{3}\right \rfloor + 2,$$

for all $g \geq 4$, which is a better upper bound for $n_g$ compared to the one obtained in Corollary \ref{cotan_g}.

\section{Gapsets with given genus, depth and small multiplicity}
\label{small}

In this section, we provide explicit formulas for the number of gapsets with fixed genus, depth and multiplicity $3$ or $4$. Recall that $\mathcal{F}(g,q,m)$ denotes the set of gapsets of genus $g$, depth $q$ and multiplicity $m$. We point out that some authors have considered studying gapsets/numerical semigroups with small multiplicities (see \cite{EF2, RPA, Karakas}).

First, we focus on the number of gapsets with fixed genus, fixed depth and multiplicity $3$ and we present a formula for $\#\mathcal{F}(g,q, 3)$.

\begin{theorem}
Let $g \geq 2$ and $q$ be positive integers. Then
$$
\#\mathcal{F}(g,q,3) = 
\begin{cases}
0, \text{ if } q < \frac{g}{2} \\
1, \text{ if } q = \frac{g}{2} \\
2, \text{ if } \frac{g+1}{2} \leq q \leq \frac{2g}{3} \\
1, \text{ if } q = \frac{2g + 1}{3} \\
0, \text{ if } q > \frac{2g + 1}{3}.
\end{cases}
$$
\label{Fgq3}
\end{theorem}

\begin{proof}
From Corollary \ref{mgq}, we conclude that $\#\mathcal{F}(g,q, 3) = 0$, if $q > \frac{2g + 2}{3}$ or $q < \frac{g}{2}$. Now we study gapsets with depth $q \in [\frac{g}{2}, \frac{2g+2}{3}]$. Let $G$ be a gapset with multiplicity $3$, depth $q$, genus $g$ and $\Kunz(G) = (k_1, k_2)$. Thus, $k_1 + k_2 = g$, $k_1, k_2 \in [1, q]$ and either $k_1 = q$ or $k_2 = q$. Furthermore, $(k_1, k_2)$ must satisfy the system (\ref{system}):

\begin{equation}
\begin{cases}
\begin{array}{lrl}
(\mathbf{I}) & 2X_1 & \geq X_2 \\
(\mathbf{II}) &  2X_2 + 1 & \geq  X_1.
\end{array}
\end{cases}
\label{m=3}
\end{equation}

There is exactly one gapset with genus $2$ and multiplicity $3$, which has Kunz coordinates $(1,1)$ and depth $1$. There are exactly two gapsets with genus $3$ and multiplicity $3$, which have Kunz coordinates $(1,2)$ and $(2,1)$ and depth $2$. Hence, the formula holds true for those cases. Now, suppose that $g \geq 4$.

\textbf{Case 1}: $\Kunz(G) = (q,q)$.

In this case, the inequalities $(\mathbf{I})$ and $(\mathbf{II})$ are satisfied. Thus, we have exactly one gapset with depth $q$ and genus $g = 2q$. 

\textbf{Case 2}: $\Kunz(G) = (q, g - q)$, with $g - q \in [1, q - 1]$.

In this case, the inequality $(\mathbf{I})$ is satisfied, and inequality $(\mathbf{II})$ is satisfied if, and only if, $q \leq \frac{2g+1}{3}$. Moreover, since $g - q \in [1, q-1]$, then $q \geq \frac{g+1}{2}$ and $q \leq g - 1$. Hence, for each pair $(g, q)$ such that $\frac{g+1}{2} \leq q \leq \min\left\{\frac{2g+1}{3}, g-1\right\} = \frac{2g+1}{3}$, there is exactly one gapset of genus $g$, depth $q$ and Kunz coordinates $(q, g - q)$.

\textbf{Case 3}: $\Kunz(G) = (g - q, q)$, with $g - q \in [1, q - 1]$.

In this case, the inequality $(\mathbf{II})$ is satisfied, and inequality $(\mathbf{I})$ is satisfied if, and only if, $q \leq \frac{2g}{3}$. Moreover, since $g - q \in [1, q-1]$, then $q \geq \frac{g+1}{2}$ and $q \leq g - 1$. Hence, for each pair $(g, q)$ such that $\frac{g+1}{2} \leq q \leq \min\left\{\frac{2g}{3}, g-1\right\} = \frac{2g}{3}$, there is exactly one gapset of genus $g$, depth $q$ and Kunz coordinates $(g - q, q)$.

By summing up the values obtained in cases 1, 2 and 3, we obtain the announced formula.

%If $k_1 = q$, then $k_2 = g - q$ and we conclude that $\frac{g}{3} \leq q \leq \frac{2g+1}{3}$. If $k_2 = q$, then $k_1 = g - q$ and we conclude that $\frac{g-1}{3} \leq q \leq \frac{2g}{3}$. Thus, there is not any gapset with multiplicity 3, genus $g$ and depth $q = \frac{2g+2}{3}$; there is exactly one gapset with multiplicity 3, genus $g$ and depth $q = \frac{2g+1}{3}$, which has Kunz coordinates $\left(\frac{2g+1}{3}, \frac{g-1}{3} \right)$; there are two gapsets with multiplicity 3, genus $g$ and depth $q$ such that $\frac{g+1}{2} \leq q \leq \frac{2g}{3}$, which have Kunz coordinates $(q, g - q)$ and $(g - q, q)$. For $q = \frac{g}{2}$, we have $(\frac{g}{2}, \frac{g}{2}) = (q, g - q) = (g - q, q)$ and, thus, there is a unique gapset with multiplicity $3$, genus $g$ and depth $\frac{g}{2}$. Notice that all the fractions have to return positive integers so that the respective Kunz coordinates are associated to a gapset.
\end{proof}

Now we provide a formula for the number of gapsets with multiplicity $4$, genus $g$ and depth $q$. The main key to obtain it is again using the Kunz coordinates of a gapset.

\begin{theorem}
Let $g \geq 7$ and $q$ be positive integers. Then
$$
\#\mathcal{F}(g,q,4) = 
\begin{cases}
\begin{array}{cl}
0, & \text{ if } q < \frac{g}{3} \\
1, & \text{ if } q = \frac{g}{3} \\
3(3q - g), & \text{ if } \frac{g}{3} < q \leq \frac{2g}{5} \\
\frac{3g+4}{5}, & \text{ if } q = \frac{2g + 1}{5} \\
\frac{3g+8}{5}, & \text{ if } q = \frac{2g + 2}{5} \\
\frac{3g+12}{5}, & \text{ if } q = \frac{2g + 3}{5} \\
\frac{3g+11}{5}, & \text{ if } q = \frac{2g + 4}{5} \text{ and } g \neq 8 \\
\lfloor \frac{g + 2q}{3} \rfloor + 2, & \text{ if } \frac{2g+4}{5} < q \leq \frac{g-1}{2} \\
\lfloor \frac{2g + 3}{3} \rfloor, & \text{ if } q = \frac{g}{2} \\
\lfloor \frac{g}{3} \rfloor, & \text{ if } q = \frac{g+1}{2} \\
0, & \text{ if } q > \frac{g + 1}{2}.
\end{array}
\end{cases}
$$
\label{Fgq4}
\end{theorem}

\begin{proof}
From Corollary \ref{mgq}, we conclude that $\#\mathcal{F}(g,q, 4) = 0$, if $q > \frac{g+1}{2}$ or $q < \frac{g}{3}$. Now we study gapsets with depth $q \in [\frac{g}{3}, \frac{g+1}{2}]$. Let $G$ be a gapset with multiplicity $4$, genus $g$, depth $q$ and $\Kunz(G) = (k_1, k_2, k_3)$. Recall that $k_1 + k_2 + k_3 = g, k_1, k_2, k_3 \in [1, q]$ and at least one of the coordinates is equal to $q$. Moreover, $(k_1, k_2, k_3)$ must satisfy the system (\ref{system}):

\begin{equation}
\begin{cases}
\begin{array}{lrl}
(\mathbf{A}) &  2X_1 & \geq X_2 \\
(\mathbf{B}) & X_1 + X_2 & \geq X_3 \\
(\mathbf{C}) & X_2 + X_3 + 1 & \geq X_1 \\
(\mathbf{D}) & 2X_3 + 1 & \geq X_2.
\end{array}
\end{cases}
\label{m=4}
\end{equation}
To avoid repetitions, we count in cases, as follows. 

\textbf{Case 1}: $\Kunz(G) = (q, q, q)$.

In this case, the inequalities $(\mathbf{A})$, $(\mathbf{B})$, $(\mathbf{C})$ and $(\mathbf{D})$ are satisfied. Thus, we have exactly one gapset with depth $q$ and genus $g = 3q$. 

\textbf{Case 2}: Exactly two of the coordinates of $\Kunz(G)$ are equal to $q$.

If exactly two of the coordinates of $\Kunz(G)$ are equal to $q$, then the other must be $g - 2q \in [1, q - 1]$. In particular, it follows that $\frac{g+1}{3} \leq q \leq \frac{g-1}{2}$. %and we shall verify which of them satisfy the system (\ref{m=4}).

\textbf{Case 2.1}: $\Kunz(G) = (q, q, g - 2q)$, with $g - 2q \in [1, q-1]$.

In this case, the inequalities $(\mathbf{A})$, $(\mathbf{B})$ and $(\mathbf{C})$ are satisfied, and inequality $(\mathbf{D})$ is satisfied if, and only if, $q \leq \frac{2g+1}{5}$. Hence, for each pair $(g, q)$ such that $\frac{g+1}{3} \leq q \leq \min\left\{\frac{g-1}{2}, \frac{2g+1}{5}\right\}$, there is exactly one gapset with genus $g$, depth $q$ and Kunz coordinates $(q, q, g - 2q)$. Observe that if $g \geq 7$, then $\min\left\{\frac{g-1}{2}, \frac{2g+1}{5}\right\} = \frac{2g+1}{5}$.

\textbf{Case 2.2}: $\Kunz(G) = (q, g - 2q, q)$, with $g - 2q \in [1, q-1]$.

In this case, the inequalities $(\mathbf{A})$, $(\mathbf{B})$,$(\mathbf{C})$ and $(\mathbf{D})$ are satisfied. Hence, for each pair $(g, q)$ such that $\frac{g+1}{3} \leq q \leq \frac{g-1}{2}$, there is exactly one gapset with genus $g$, depth $q$ and Kunz coordinates $(q, g - 2q, q)$.

\textbf{Case 2.3}: $\Kunz(G) = (g - 2q, q, q)$, with $g - 2q \in [1, q-1]$.

In this case, the inequalities $(\mathbf{B})$, $(\mathbf{C})$ and $(\mathbf{D})$ are satisfied, and inequality $(\mathbf{A})$ is satisfied if, and only if, $q \leq \frac{2g}{5}$. Hence, for each pair $(g, q)$ such that $\frac{g+1}{3} \leq q \leq \min\left\{\frac{g-1}{2}, \frac{2g}{5}\right\}$, there is exactly one gapset with genus $g$, depth $q$ and Kunz coordinates $(g - 2q, q, q)$. Observe that if $g \geq 5$, then $\min\left\{\frac{g-1}{2}, \frac{2g}{5}\right\} = \frac{2g}{5}$.

\textbf{Case 3}: Exactly one of the coordinates of $\Kunz(G)$ is equal to $q$.

If exactly one of the coordinates is equal to $q$, then the others must lie on $[1, q-1]$ and their sum must be $g - q$. 

\textbf{Case 3.1}: $\Kunz(G) = (q, k_2, k_3)$, with $k_2 + k_3 = g - q$ and  $k_2, k_3 \in [1, q-1]$.

Observe that $g - q = k_2 + k_3 \leq 2q - 2$. Thus, if $q < \frac{g+2}{3}$, then there is not any solution. Now, suppose that $\frac{g+2}{3} \leq q \leq \frac{g+1}{2}$. In this case, the inequalities ($\mathbf{A}$), ($\mathbf{B}$) and ($\mathbf{C}$) are verified. Hence, we only have to check how many integer solutions $(k_2, k_3) \in [1, q-1] \times [1, q-1]$ are there for

\begin{equation}
\begin{cases}
\begin{array}{rl}
2X_3 + 1 & \geq X_2 \\
X_2 + X_3 & = g - q.
\end{array}
\end{cases}
\label{m=4,k1=q}
\end{equation}
The integer solutions of (\ref{m=4,k1=q}) lie in the line $X_2 + X_3 = g - q$ and we can count them by analysing three sub-cases, which depend on the relative position of the square $[1,q-1]^2$ and the lines $X_2 = 2X_3 + 1$ and $X_2 + X_3 = g - q$. For the first one, let $\frac{g+2}{3} \leq q \leq \frac{2g+4}{5}$. Then there are $3q - g - 1$ solutions given by $\left(j, g - q - j\right)$, where $j \in \left [g - 2q + 1, q - 1 \right]$. For the second sub-case, let $\frac{2g+4}{5} < q \leq \frac{g}{2}$. Then there are $\left \lfloor \frac{4q - g + 1}{3} \right \rfloor$ solutions given by $\left(j, g - q - j\right)$, where $j \in \left [g - 2q + 1, \frac{2g - 2q + 1}{3} \right]$. For the third sub-case, let $q = \frac{g+1}{2}$. Then there are $\left \lfloor \frac{g}{3} \right \rfloor$ solutions given by $\left(j, \frac{g-1}{2} - j\right)$, where $j \in \left [1, \frac{g}{3} \right]$.

\textbf{Case 3.2}: $\Kunz(G) = (k_1, q, k_3)$, with $k_1 + k_3 = g - q$ and  $k_1, k_3 \in [1, q-1]$.

Observe that $g - q = k_1 + k_3 \leq 2q - 2$. Thus, if $q < \frac{g+2}{3}$, then there is not any solution. Now, suppose that $\frac{g+2}{3} \leq q \leq \frac{g+1}{2}$. In this case, the inequalities ($\mathbf{B}$) and ($\mathbf{C}$) are verified. The inequalities ($\mathbf{A}$) and ($\mathbf{D}$) become $X_1 \geq \frac{q}{2}$ and $X_3 \geq \frac{q-1}{2}$. Hence, we only have to check how many integer solutions $(k_1, k_3) \in \left[\frac{q}{2}, q-1\right] \times \left[\frac{q-1}{2}, q-1\right]$ are there for $X_1 + X_3 = g - q$. 

We count its integer solutions by analysing three sub-cases, which depend on the relative position of the rectangle  $\left[\frac{q}{2}, q-1\right] \times \left[\frac{q-1}{2}, q-1\right]$ and the line $X_1 + X_3 = g - q$. For the first sub-case, let $\frac{g+2}{3} \leq q \leq \frac{2g+2}{5}$. Then there are $3q - g - 1$ solutions given by $\left(j, g - q - j\right)$, where $j \in \left [g - 2q + 1, q-1 \right]$. For the second sub-case, let $\frac{2g+2}{5} < q \leq \frac{g}{2}$. Then there are $g-2q+1$ solutions given by $\left(j, g - q - j\right)$, where $j \in \left [\frac{q}{2}, g - \frac{3q-1}{2} \right]$. For the third sub-case, let $q = \frac{g+1}{2}$. In this sub-case, $\frac{g-1}{2} = g - q = k_1+k_3 \geq \frac{q}{2} + \frac{q-1}{2} = \frac{g}{2}$, which is a contradiction. Hence, there is not integer any solution in this sub-case.

 %Here we observe that if $q < \frac{g+2}{3}$ or $q = \frac{g+1}{2}$, then there is not any integer point that satisfies the system (\ref{m=4,k2=q}).

%Let $k_2 = q$ and $1 \leq k_1, k_3 \leq q-1$, where $k_1 + k_3 = g - q$. If $q \leq \frac{g+1}{3}$, then there is not any integer point. If $\frac{g+2}{3} \leq q \leq \frac{2g+2}{5}$, then there are $3q - g - 1$ integer points. If $\frac{2g+2}{5} < q \leq \frac{g}{2}$, then there are $g - 2q + 1$ integer points. 

\textbf{Case 3.3}: $\Kunz(G) = (k_1, k_2, q)$, with $k_1 + k_2 = g - q$ and  $k_1, k_2 \in [1, q-1]$.

Observe that $g - q = k_1 + k_2 \leq 2q - 2$. Thus, if $q < \frac{g+2}{3}$, then there is not any solution. Now, suppose that $\frac{g+2}{3} \leq q \leq \frac{g+1}{2}$. In this case, the inequalities ($\mathbf{C}$) and ($\mathbf{D}$) are verified and the inequality ($\mathbf{B}$) is verified if, and only if, $q \leq \frac{g}{2}$. Hence, we only have to check how many integer solutions $(k_1, k_2) \in [1, q-1] \times [1, q-1]$ are there for

\begin{equation}
\begin{cases}
\begin{array}{rl}
2X_1 & \geq X_2 \\
X_1 + X_2 & = g - q.
\end{array}
\end{cases}
\label{m=4,k3=q}
\end{equation}

The integer solutions of (\ref{m=4,k3=q}) lie in the line $X_1+X_2 = g - q$ and we can count them by analysing two sub-cases, which depend on the relative position of the square $[1,q-1]^2$ and the lines $X_2 = 2X_1$ and $X_1 + X_2 = g - q$. For the first one, let $\frac{g+2}{3} \leq q \leq \frac{2g+3}{5}$. Then there are $3q - g - 1$ solutions given by $\left(g - q - j, j\right)$, where $j \in \left [g - 2q + 1, q - 1 \right]$. For the second sub-case, let $\frac{2g+3}{5} < q \leq \frac{g}{2}$. Then there are $\left \lfloor \frac{4q - g}{3} \right \rfloor$ solutions given by $\left(g - q - j, j\right)$, where $j \in \left [g - 2q + 1, \frac{2g - 2q}{3} \right]$.

%Let $k_3 = q$ and $1 \leq k_1, k_2 \leq q-1$, where $k_1 + k_2 = g - q$. If $q \leq \frac{g+1}{3}$, then there is not any integer point. If $\frac{g+2}{3} \leq q \leq \frac{2g+3}{5}$, then there are $3q - g - 1$ integer points. If $\frac{2g+3}{5} < q \leq \frac{g}{2}$, then there are $\left \lfloor \frac{4q - g}{3} \right \rfloor$ integer points.

Now, for $g \geq 7$, we only have to sum the values obtained to complete the proof. Here we notice that the case $g = 8$, $q = 4$ deserves a special care, because $\frac{2 \cdot 8 + 4}{5} = \frac{8}{2}$. In particular, we have to use the formula with $q = \frac{g}{2}$ and not with $q = \frac{2g+4}{5}$, since there is not any gapset in the Case 2.2 with $g = 8$ and $q = 4$. Moreover, we used the fact that for an integer $a$, it holds true that $\left \lfloor \frac{a}{3} \right \rfloor + \left \lfloor \frac{a+1}{3} \right \rfloor =\left \lfloor \frac{2a}{3} \right \rfloor$.
\end{proof}

\begin{remark}
If $g \leq 2$, then $\#\mathcal{F}(g,q,4) = 0$. If $3 \leq g \leq 6$, then the values $\#\mathcal{F}(g,q,4)$ can be obtained with the method presented in the proof of Theorem \ref{Fgq4} and also with the package \texttt{numericalsgps} \cite{GAP} in GAP. We present the values of $\#\mathcal{F}(g,q,4)$, for $g \leq 12$ in Table \ref{gq4}.
\end{remark}

\setlength{\tabcolsep}{0.3em} % for the horizontal padding
{\renewcommand{\arraystretch}{1.3}% for the vertical padding
\begin{table}
\begin{tabular}{|c|c c c c c c c c c c|}
 \hline
\diagbox[height=1.1cm]{$q$}{$g$} & 3 & 4 & 5 & 6 & 7 & 8 & 9 & 10 & 11 & 12  \\
\hline
1 & 1 &  &  &  &  & & & & &   \\
\hline
2 &  & 3 & 3 & 1 &  & & & & &    \\
\hline
3 &  &  & 1 & 5 & 5 & 3 & 1 & & &  \\
\hline
4 &  &  &  & &  2 & 6 & 7 & 6 & 3 & 1 \\
\hline
5 &  &  &  &  &  & & 3 & 7 & 9 & 8 \\
\hline
6 &  &  &  &  &  & & & & 3 & 9 \\   
\hline \hline
$N(4,g)$ & 1 & 3 & 4 & 6 & 7 & 9 & 11 & 13 & 15 & 18 \\   
\hline
\end{tabular}
\vspace{0.3cm}
\caption{A few values for $\#\mathcal{F}(g, q, 4)$}
\label{gq4}
\end{table}
}

Next, we present Table \ref{tabela3}, which is constructed using the formulas given in Theorems \ref{Fgq3} and \ref{Fgq4} and Example \ref{upper_casos}.

\begin{table}[h]
\centering
\begin{tabular}{|c|c|c|c|c|c|}
  \hline
   $g$ & $n_g$ & UB Ex. \ref{upper_casos} ($M=4$) & UB Ex. \ref{upper_casos} ($M=3$)  &  UB Ex. \ref{upper_casos} ($M=2$)  & UB Cor. \ref{cotan_g}  \\
     \hline                               
%     0 & 1 &  & * & * & * \\
%     \hline
     1 & 1 & 2 & 2 & 2 & 1 \\
     \hline
     2 & 2 & 3 & 2 & 3 & 2 \\
     \hline
     3 & 4 & 6 & 4 & 4 & 4 \\
     \hline
     4 & 7 & 8 & 7 & 8 & 8 \\
     \hline
     5 & 12 & 12 & 14 & 16 & 16 \\
     \hline
     6 & 23 & 28 & 27 & 30 & 32 \\
     \hline
     7 & 39 & 50 & 58 & 62 & 64 \\
     \hline
     8 & 67 & 112 & 111 & 126 & 128 \\
     \hline
     9 & 118 & 216 & 239 & 249 & 256 \\
     \hline
     10 & 204 & 413 & 468 & 505 & 512 \\
     \hline
     \end{tabular}
   \vspace{0.3cm}
   \caption{Upper bounds (UB) for $n_g$ using Example \ref{upper_casos} and Corollary \ref{cotan_g}}
   \label{tabela3}
\end{table}

To end up, we provide closed formulas for $\#\mathcal{F}(g,q)$, for some values of $q$, when $g$ is fixed, which are illustrated in Table \ref{g,q}.

\pagebreak

\begin{theorem}
Let $g$ be a positive integer. Then

$\begin{array}{lll}
(1) &  \#\mathcal{F}(g,1) = 1 & \\ 
(2) & \#\mathcal{F}(g,2) = \F_{g+1} - 1 & \\
(3) & \#\mathcal{F}(g,q) = \lfloor \frac{g + 2q}{3} \rfloor + 2, & \text{ if } \frac{2g+4}{5} < q \leq \frac{g-1}{2} \text{ and }  g \geq 23 \\
(4) & \#\mathcal{F}(g,\frac{g}{2}) = \left \lfloor \frac{2g}{3} \right \rfloor + 2, & \text{ if } g \equiv 0 \pmod{2} \text{ and } g \geq 8 \\
(5) & \#\mathcal{F}(g,\frac{g + 1}{2}) = \left \lfloor \frac{g}{3} \right \rfloor + 2, & \text{ if } g \equiv 1 \pmod{2}, g \geq 5 \\
(6) & \#\mathcal{F}(g,q) = 2, & \text{ if } \frac{g+2}{2} \leq q \leq \frac{2g}{3} \\
(7) & \#\mathcal{F}(g,\frac{2g + 1}{3}) = 1, & \text{ if } g \equiv 1 \pmod{3}, g \geq 4 \\
(8) & \#\mathcal{F}(g,\frac{2g + 2}{3}) = 0, & \text{ if } g \equiv 2 \pmod{3}, g \geq 5 \\
(9) & \#\mathcal{F}(g,q) = 0, & \text{ if } \left \lceil \frac{2g}{3} \right \rceil + 1 \leq q \leq g - 1 \\
(10) & \#\mathcal{F}(g,g) = 1 & \\
(11) & \#\mathcal{F}(g,q) = 0, & \text{ if } q > g
\end{array}$
\label{valores}
\end{theorem}
%  \item $|\mathcal{F}(g,3)| = n'_g - \F_{g+1}$

\begin{proof}
Items (1) and (10) follow from the facts that the only gapsets that have depth equals one are the ordinary gapsets and the only gapsets that have depth equals to the genus are the hyperelliptic gapsets, respectively. Item (11) follows from Proposition \ref{dep}. Item (2) follows from Proposition \ref{q=2}. Item (3) follows from Theorem \ref{Fgq4}. Items (4) and (5) follow from Theorems \ref{Fgq3} and \ref{Fgq4}. Items (6), (7) and (8) follow from Theorem \ref{Fgq3}. Item (9) follows from Corollary \ref{zeros}.

%\begin{enumerate}
%  \item  The only gapsets that have depth equals one are the ordinary gapsets.
%  \item It follows from Proposition \ref{q=2}.
%  \item It follows from Theorem \ref{Fgq4}.
%  \item It follows from Theorems \ref{Fgq3} and \ref{Fgq4}.
%  \item It follows from Theorems \ref{Fgq3} and \ref{Fgq4}.
%  \item It follows from Theorem \ref{Fgq3}.
%  \item It follows from Theorem \ref{Fgq3}.
%  \item It follows from Theorem \ref{Fgq3}.
%  \item It follows from Corollary \ref{zeros}.
%  \item The only gapsets that have depth equals to the genus are the hyperelliptic gapsets.
%  \item It follows from Proposition \ref{dep}.
%\end{enumerate}
\end{proof}

\setlength{\tabcolsep}{0.22em} % for the horizontal padding
{\renewcommand{\arraystretch}{1.4}% for the vertical padding
\begin{table}
\begin{tabular}{|c|c c c c || c || c c c c c c c c c c c c c c c| c|}
 \hline
\diagbox[height=1.1cm]{$g$}{$q$} & $0$ & $1$ & $2$ & $3$ & $n'_g$ &$4$ & $5$ & $6$ & $7$ & 8 & 9 & 10 & 11 & 12 & 13 & 14 & 15 & 16 & 17 & 18 & $n_g$ \\
\hline
0 & \textbf{1} &  &  &  & 1 &  &  &  &  &  &  &  &  &  &  &  &  &  &  &  & 1 \\ 
\hline
1 &  & \textbf{1} &  &  & 1 &  &  &  &  &  &  &  &  &  &  &  &  &  &  &  & 1 \\
\hline
2 &  & \textbf{1} & \textbf{1} &  & 2 &  &  &  &  &  &  &  &  &  &  &  &  &  &  &  & 2 \\
\hline
3 &  & \textbf{1} & \textbf{2} & \textbf{1}  & 4 &  &  &  &  &  &  &  &  &  &  &  &  &  &  &  & 4 \\
\hline
4 &  & \textbf{1} & \textbf{4} & \textbf{1} & 6 & \textbf{1} &  &  &  &  &  &  &  &  &  &  &  &  &  &  & 7 \\
\hline
5 &  & \textbf{1} & \textbf{7} & \textbf{3} & 11 & \textbf{0} & \textbf{1} &  &  &  &  &  &  &  &  &  &  &  &  &  & 12 \\
\hline
6 &  & \textbf{1} & \textbf{12} & 7 & 20 & \textbf{2} & \textbf{0} & \textbf{1} &  &  &  &  &  &  &  &  &  &  &  &  & 23 \\
\hline
7 &  & \textbf{1} & \textbf{20} & 12 & 33 & \textbf{4} & \textbf{1} & \textbf{0} & \textbf{1} &  &  &  &  &  &  &  &  &  &  &  & 39 \\
\hline
8 &  & \textbf{1} & \textbf{33} & 23 & 57 & \textbf{7} & \textbf{2} & \textbf{0} & \textbf{0} & \textbf{1} &  &  &  &  &  &  &  &  &  &  & 67 \\
\hline
9 &  & \textbf{1} & \textbf{54} & 44 & 99 & 11 & \textbf{5} & \textbf{2} & \textbf{0} & \textbf{0} & \textbf{1} &  &  &  &  &  &  &  &  &  & 118 \\
\hline
10 &  & \textbf{1} & \textbf{88} & 79 & 168 & 24 & \textbf{8} & \textbf{2} & \textbf{1} & \textbf{0} & \textbf{0} & \textbf{1} &  &  &  &  &  &  &  &  & 204 \\
\hline
11 &  & \textbf{1} & \textbf{143} & 143 & 287 & 37 & 11 & \textbf{5} & \textbf{2} & \textbf{0} & \textbf{0} & \textbf{0} & \textbf{1} &  &  &  &  &  &  &  & 343 \\
\hline
12 &  & \textbf{1} & \textbf{232} & 254 & 487 & 71 & 19 & \textbf{10} & \textbf{2} & \textbf{2} & \textbf{0} &  \textbf{0} & \textbf{0} & \textbf{1} &  &  &  &  &  &  & 592 \\
\hline
13 &  & \textbf{1} & \textbf{376} & 447 & 824 & 124 & 33 & 10 & \textbf{6} & \textbf{2} & \textbf{1} & \textbf{0} & \textbf{0} & \textbf{0} & \textbf{1} &  &  &   &  &   & 1001 \\
\hline
14 &  & \textbf{1} & \textbf{609} & 785 & 1395 & 209 & 57 & 16 & \textbf{11} & \textbf{2} & \textbf{2} & \textbf{0} & \textbf{0} & \textbf{0} &  \textbf{0} & \textbf{1} &  &  &  &  & 1693 \\
\hline
15 &  & \textbf{1} & \textbf{986} & 1364 & 2351 & 353 & 104 & 26 & 11 & \textbf{7} & \textbf{2} & \textbf{2} & \textbf{0} & \textbf{0} & \textbf{0} & \textbf{0} & \textbf{1} &  &  &  & 2857 \\
\hline
16 &  & \textbf{1} & \textbf{1596} & 2357 & 3954 & 612 & 158 & 48 & 16 & \textbf{12} & \textbf{2} & \textbf{2} & \textbf{1} & \textbf{0} & \textbf{0} & \textbf{0} & \textbf{0} & \textbf{1} &  &  & 4806 \\
\hline
17 &  & \textbf{1} & \textbf{2583} & 4052 & 6636 & 1028 & 254 & 79 & 23 & 13 & \textbf{7} & \textbf{2} & \textbf{2} & \textbf{0} & \textbf{0} & \textbf{0} & \textbf{0} & \textbf{0} & \textbf{1} &  & 8045 \\
\hline
18 &  & \textbf{1} & \textbf{4180} & 6935 & 11116 & 1739 & 409 & 132 & 37 & 13 & \textbf{14} & \textbf{2} & \textbf{2} & \textbf{2} & \textbf{0} & \textbf{0} & \textbf{0} & \textbf{0} & \textbf{0} & \textbf{1} & 13467 \\
\hline
\end{tabular}
\vspace{0.3cm}
\caption{A few values for $\#\mathcal{F}(g, q)$. The entries in bold correspond to values that can be obtained using Theorem \ref{valores}}
\label{g,q}
\end{table}
}

\section{Further questions and concluding remarks}

In sections 3 and 4, we considered $m$-extensions and we observed that there is a bijective map between the set of all $m$-extensions of genus $g$ and $\N^{m-1}$, by considering the pseudo Kunz coordinates. In future works, it could be of interest investigating properties of $m$-extensions. For instance, one can show that if $A$ is an $m$-extension of genus $g$ and conductor $c$, then $g+1 \leq c \leq \left \lfloor \frac{(g+2)^2}{4} \right \rfloor$, which is a property that is not shared by gapsets. In Remark \ref{BA_mext}, we proved a version of Bras Amorós' conjecture. Natural questions that arise are how could versions of the central problems in numerical semigroup theory (Wilf's conjecture, Frobenius' problem, etc.) be formulated and how they can be proved for some families of $m$-extensions or disproved.

In section 5, we obtained a lower bound for the number $n'_g$, by considering $t$-uples that return Kunz coordinates of gapsets with depth at most 3. An interesting question is how could we construct other families of gapsets of depth at most $3$, in order to obtain a better lower bound for $n'_g$ involving well known sequences.

In section 6, we obtained a sequence of upper bounds for $n_g$. A natural question is how can we improve them, by refining the presented method or combining it with some other methods. We also expect to improve the hypothesis $g \gg 0$ of the Corollary \ref{cor610}.

In section 7, we provided closed formulas for $
\#\mathcal{F}(g,q,3), \#\mathcal{F}(g,q,4)$ and for $\#\mathcal{F}(g,q)$, in some particular cases. One could ask for formulas for $\#\mathcal{F}(g,q,m)$, for $m > 4$ and for $\#\mathcal{F}(g,q)$ for arbitrary $g$ and $q$.

%{\bf Acknowledgment.} The authors thank to the anonymous referee and to the editor for their careful corrections, suggestions and comments that allowed to improve the last version of the paper.

\end{document}